\documentclass{amsart}
\usepackage{array}
\usepackage{amsmath}
\usepackage{latexsym}  
\usepackage{mathrsfs}
\usepackage[colorlinks]{hyperref}
\usepackage{xypic}
\usepackage{tikz,tikz-cd}
\usetikzlibrary{math}

\newtheorem{theorem}{Theorem}[section]
\newtheorem{lemma}[theorem]{Lemma}
\newtheorem{thm}[theorem]{Theorem}
\newtheorem{prop}[theorem]{Proposition}
\newtheorem{cor}[theorem]{Corollary}

\theoremstyle{definition}
\newtheorem{definition}[theorem]{Definition}
\newtheorem{example}[theorem]{Example}

\theoremstyle{remark}
\newtheorem{remark}[theorem]{Remark}
\newtheorem{rmk}[theorem]{Remark}

\numberwithin{equation}{section}

\begin{document}
\title{ Positivity of $\Delta$-genera for connected polarized demi-normal schemes}
\author{Jingshan Chen}
\address{school of mathematics and statistics, Hubei Minzu University}
\email{chjingsh@hbmzu.edu.cn}
\author{Yongchang Chen}
\address{department of mathematics, University of Houston }
\email{ychen224@CougarNet.uh.edu}

\thanks{Thanks to Prof. Jinxing Cai and Prof. Wenfei Liu.}





\begin{abstract}

In this paper, we show that the  $\Delta$-genus $\Delta(X,\mathcal{L})\ge 0$ for any  connected polarized demi-normal scheme $(X,\mathcal{L})$ over an algebraically closed field $k$. As an application, we obtain $\Delta(X,I(K_X+\Lambda))\ge 0$ for any KSBA stable log scheme $(X,\Lambda)$, where  $I$ is the Cartier index of $K_X+\Lambda$. We also construct examples of KSBA stable log schemes with $I=1$ and $\Delta(X,K_X+\Lambda)= 0$, which shows the inequality is sharp when $I=1$.


\end{abstract}

\maketitle

\section{Introduction}

We work over an algebraically closed field $k$. 
A scheme $X$ is demi-normal if it is $S_2$ and at worst nodal at any generic point of codimension 1. The term "demi-normal" was coined by Koll\'ar in \cite{KollarSMMP} to define KSBA stable (log) schemes.  A KSBA stable log scheme $(X,\Lambda)$ is a demi-normal scheme $X$ with a boundary divisor $\Lambda$ such that it admits only slc singularities and $K_X+\Lambda$ is an ample $\mathbb{Q}$-Cartier divisor. A KSBA stable scheme is a KSBA stable log scheme with empty boundary. 
KSBA stable (log) schemes are the fundamental objects to construct the compactifications of moduli spaces of smooth varieties of general type. 

Given a demi-normal scheme $X$, it is natural to consider a polarization, i.e., an ample invertible sheaf $\mathcal{L}$ on $X$. 
For a KSBA stable log scheme $(X,\Lambda)$, $I(K_X+\Lambda)$ is a polarization on it, where $I$ is the Cartier index of $K_X+\Lambda$. 
Fujita introduced an invariant, $\Delta$-genus $\Delta(X, \mathcal{L}):=(\mathcal{L})^{\dim X}-h^0(X, \mathcal{L})+\dim X$ for polarized varieties. He showed that $\Delta(X,\mathcal{L})\ge 0$ for any irreducible polarized variety (see \cite[Theorem (1.4.2)]{Fuj90}). In this paper, we show that 
\begin{thm}[see Theorem \ref{mainthm}]
	For any connected polarized demi-normal scheme $(X,\mathcal{L})$, we have
	$$\Delta(X,\mathcal{L})\ge 0.$$
\end{thm}
As a result, we obtain $\Delta(X,I(K_X+\Lambda))\ge 0$ for any KSBA stable log scheme $(X,\Lambda)$. In particular,  when $(X,\Lambda)$ is Gorenstein, i.e., $I=1$, we have $\Delta(X,K_X+\Lambda)\ge 0$, which is an inequality of Noether type. We remark that when $\mathrm{dim} X=2$ and the boundary divisor $\Lambda$ is reduced,  the inequality $\Delta(X, K_X+\Lambda)=(K_X+\Lambda)^2-p_g(X,\Lambda)+2\ge 0$ is just the stable log Noether inequality established  by Liu and Rollenske in \cite{LR13}. We note that this paper is strongly inspired by Liu and Rollenske's work. 
 
We also characterize those connected polarized demi-normal schemes with $(X,\mathcal{L})=0$ in Theorem \ref{delta0schemes}. They are trees of varieties $X_i$'s with $\Delta(X_i,\mathcal{L}|_{X_i})=0$ glued along hyperplanes. 
Then we construct examples of KSBA stable log scheme $(X,\Lambda)$ with $I=1$ and  $\Delta(X,K_X+\Lambda)=0$  in Example \ref{exmp}, which implies the inequality $\Delta(X,K_X+\Lambda)\ge 0$ is sharp for KSBA stable log schemes with $I=1$. 

The paper is organised as follows. In \textsection 2, we state some facts about polarized demi-normal schemes. In \textsection 3, we review Fujita's work on $\Delta$-genus of polarized varieties. In \textsection 4, we prove that $\Delta(X,\mathcal{L})\ge 0$ for connected polarized demi-normal schemes. 


\subsection*{Acknowledgments:}
We would like to express our sincere gratitude to Prof. Jinxing Cai and Prof. Wenfei Liu, for their invaluable guidance and instruction. 
Additionally, we would like to extend our thanks to the anonymous reviewers for their insightful comments and suggestions. Their insightful comments and suggestions have been instrumental in improving this paper. 
\subsection{Notations and conventions}
We work with schemes defined over an algebraically closed field $k$. They are assumed to be proper and of finite type over $k$. 
\begin{itemize}
	\item  A variety is an integral scheme of finite type over $k$.
	\item  We assume that a normal singular variety always admits a resolution of singularities.
	\item  By abuse of notation, we sometimes do not distinguish between a Cartier divisor $D$ and its associated invertible sheaf $\mathcal{O}_X(D)$.
	\item  We use '$\equiv$' to denote linear equivalence relation of divisors.
	\item  If $D$ is a Cartier divisor on $X$, then we denote by $\Phi_{|D|}:X\dashrightarrow \mathbb{P}:=|D|^*$  the rational map defined by the linear system $|D|$.
	\item  A hyperplane $H$ on a scheme $X$ with respect to $\mathcal{O}_X(1)$ is a subscheme isomorphic to $\mathbb{P}^{\dim X-1}$ such that $\mathcal{O}_X(1)|_{H}\cong \mathcal{O}_{\mathbb{P}^{\dim X-1}}(1)$. 
	\item We use $\mathrm{Bs} |L|$ to denote the set-theoretic intersection of all the members of $|L|$.
	\item We use  $(\mathcal{L})^n$ to denote the $n$-th self-intersection of the associated Cartier divisor of $ \mathcal{L}$ in the Chow ring. 
\end{itemize}
\section{Demi-normal schemes}

A scheme $X$ is {\bf demi-normal} if it is $S_2$ and at worst nodal at any generic point of codimension 1. The schemes we consider are always assumed to be connected. 

Let $X$ be a demi-normal scheme and  let $\pi: \bar X \to X$ denote its normalization morphism. The conductor ideal
$ \mathcal{H}om_{\mathcal{O}_X}(\pi_*\mathcal{O}_{\bar{X}}, \mathcal{O}_X)$
is an ideal sheaf on both $X$ and $\bar{X}$ and hence defines subschemes
$D\subset X \text{ and } \bar D\subset \bar X,$
both reduced and of pure codimension 1; we often refer to $D$ as the non-normal locus of $X$.

A {\bf polarized demi-normal scheme} is a pair $(X, \mathcal{L})$, where $X$ is a proper demi-normal scheme and $\mathcal{L}$ is an ample invertible sheaf on $X$. 

\begin{definition}	
	The {\bf $\Delta$-genus} of a polarized demi-normal scheme $(X,\mathcal{L})$ is defined as $$\Delta(X, \mathcal{L}):=(\mathcal{L})^{\dim X}-h^0(X, \mathcal{L})+\dim X.$$
\end{definition}
\begin{remark}
	Fujita defined $\Delta$-genus for polarized varieties (cf. Definition \ref{FujitaDeltagenus}). We generalize it to polarized demi-normal schemes. 
\end{remark}

Let $(X,\mathcal{L})$ be a connected polarized proper demi-normal scheme. Assume that $X$ can be decomposed into two connected components, i.e.,  $X=X_1\cup X_2$, where $X_1$ and $X_2$ are connected. The subscheme $C:=X_1\cap X_2$ is contained in the non-normal locus $D$ of $X$ and is of codimension 1. We call $C$ the {\em connecting subscheme} of $X_1$ and $X_2$. 
We have the  Mayer-Vietoris exact sequence:
\begin{align*}
0\to \mathcal{L} \to  \mathcal{L}|_{X_1}\oplus\mathcal{L}|_{X_2}\to \mathcal{L}|_{C}\to 0.
\end{align*}
Taking the associated long exact sequence, we obtain
\begin{alignat}{3}
0 \to H^0(X,\mathcal{L}) \to H^0(X_1,\mathcal{L}|_{X_1})\oplus H^0(X_2,\mathcal{L}|_{X_2}) \xrightarrow{\phi} H^0(C,\mathcal{L}|_{C}).
\end{alignat}

The morphism $\phi$ is defined by $(\mathcal{R}_{X_1\to C}, -\mathcal{R}_{X_2\to C})$, where $\mathcal{R}_{X_i\to C}$ is the restriction map.
Denote $\dim\mathrm{im}\,\mathcal{R}_{X_i\to C}$ by $r_{X_i\to C}(\mathcal{L}|_{X_i})$ or $r_{X_i\to C}$ for simplicity.
We have
\begin{equation}\label{sectiongluing}
\begin{split}
h^0(X, \mathcal{L})&=
h^0(X_1,\mathcal{L}|_{X_1})+h^0(X_2,\mathcal{L}|_{X_2})-\dim\mathrm{im}\,\phi\\
&\le h^0(X_1,\mathcal{L}|_{X_1})+h^0(X_2,\mathcal{L}|_{X_2})-\max\{r_{X_1\to C},r_{X_2\to C}\}.
\end{split}
\end{equation}

\section{Fujita's $\Delta$-genus of polarized varieties}

In this section we recall some definitions and results of Fujita from \cite[Chapter 1]{Fuj90}. Note that the ground field $k$ is an algebraically closed field $k$ of characteristic $\ge 0$ in \cite{Fuj90}. 

Let $X$ be a variety over $k$ of dimension $n$ and $\mathcal{L}$ is ample invertible sheaf on it. 
Such a pair $(X,\mathcal{L})$ is called a {\bf polarized variety}.

\begin{definition}(\cite[(1.2.0)]{Fuj90})
An element $D\in|\mathcal{L}|$ is called a {\bf rung} of $(X,\mathcal{L})$ if  $D$ is reduced and irreducible as a subscheme of $X$.
\end{definition}
\begin{remark}
If $D$ is a rung of $(X,\mathcal{L})$, then $(D,\mathcal{L}|_D)$ is a polarized variety of dimension $n-1$. 
\end{remark}

Let $\chi(t\mathcal{L})$ be the Euler-Poincar\'e characteristic of $\mathcal{L}^t$, which is a polynomial in $t$ of degree $n$. We put
$$ \chi(t\mathcal{L})=\sum\limits^{n}_{j=0}\chi_j(X,\mathcal{L})\frac{t^{[j]}}{j!},
$$
where $t^{[j]}=t(t+1)...(t+j-1)$ for $j\ge1$ and $t^{[0]}=1$. 

\begin{definition}(\cite[(1.2.1)]{Fuj90})
	The {\bf degree} of $(X,\mathcal{L})$ is defined as $d(X,\mathcal{L}):=\chi_n(X,\mathcal{L})$, which equals $(\mathcal{L})^n$ by Riemann Roch theorem.
	
	The {\bf sectional genus } of $(X,\mathcal{L})$ is defined as $g(X,\mathcal{L}):=1-\chi_{n-1}(X,\mathcal{L})$.
\end{definition}

\begin{rmk}
If $\dim X=1$, $g(X,\mathcal{L})=h^1(X,\mathcal{O}_X)$ is just the arithmetic genus of the curve $X$.
If  $\dim X\ge2$ and  $X$ is non-singular, by Riemann Roch theorem we have $g(X,\mathcal{L})=(K_X+(n-1)\mathcal{L})\cdot (\mathcal{L})^{n-1}/2+1$.
\end{rmk}
\begin{prop}(\cite[(1.2.1)]{Fuj90})

Let $D$ be a rung of $(X,\mathcal{L})$. Then $$\chi_r(D,\mathcal{L}|_D)=\chi_{r+1}(X,\mathcal{L})  \quad \text{for}~r\ge0.$$

In particular, $g(D,\mathcal{L}|_D)=g(X,\mathcal{L})$.
\end{prop}

\begin{definition}(Fujita's $\Delta$-genus, cf. \cite[(1.2.2)]{Fuj90})\label{FujitaDeltagenus}
	
	The {\bf $\Delta$-genus} of a polarized variety $(X,\mathcal{L})$ is defined as $$\Delta(X, \mathcal{L}):=d(X,\mathcal{L})-h^0(X, \mathcal{L})+\dim X=(\mathcal{L})^{\dim X}-h^0(X, \mathcal{L})+\dim X.$$
\end{definition}

\begin{definition}(\cite[(1.3.1)]{Fuj90})
	
	A sequence $D_1\subset D_2\subset ... \subset D_n\subset X$ of subvarieties of $X$ is called a {\bf ladder} of $(X,\mathcal{L})$ if $D_{j}$ is a rung of $(D_{j+1},\mathcal{L}_{j+1})$ for each $j\ge1$, where $\mathcal{L}_{j+1}=\mathcal{L}|_{D_{j+1}}$.
\end{definition}

\begin{prop}
(\cite[Proposition (1.3.4)]{Fuj90})\label{sectionalgenus}

Let $(X,\mathcal{L})$ be a polarized variety with $g(X,\mathcal{L})=0$. Suppose that $(X,\mathcal{L})$ has a ladder $\{D_j\}$ such that each rung $D_j$ (including $X$ itself) is a normal variety. 

Then $\Delta(X,\mathcal{L})=0$.
\end{prop}

\begin{thm}(\cite[Theorem (1.3.5)]{Fuj90})\label{bpf}
Let $(X,\mathcal{L})$ be a polarized variety having a ladder.
Assume that $g(X,\mathcal{L})\ge \Delta(X,\mathcal{L})$.

Then
$\mathrm{Bs}|\mathcal{L}|=\emptyset$, if $d(X,\mathcal{L})  \ge 2\Delta(X,\mathcal{L})$.

\end{thm}

\begin{thm}(\cite[Theorem (1.4.2)]{Fuj90})\label{fujita}
	
	Let $(X, \mathcal{L})$ be a polarized variety, then  $\Delta(X, \mathcal{L})\ge \dim \mathrm{Bs}|\mathcal{L}|+1$, where $\mathrm{Bs}|\mathcal{L}|$ is the base locus of $|\mathcal{L}|$ and $\dim \emptyset$ is defined to be $-1$.
	
 	In particular, $$\Delta(X, \mathcal{L})\ge 0.$$ 
 	Moreover, $\mathrm{Bs}|\mathcal{L}|=\emptyset$ if $\Delta(X, \mathcal{L})=0$.
	
\end{thm}

\begin{thm}(\cite[Corollary (1.4.12)]{Fuj90})\label{veryample}
	
	Let $(X, \mathcal{L})$ be a polarized variety with $\Delta(X, \mathcal{L})= 0$.
	Then $\mathcal{L}$ is very ample, and $X$ is normal. 
\end{thm}
\begin{prop}(\cite[(1.4.13)]{Fuj90})\label{ladder}
	
	Let $(X, \mathcal{L})$ be a polarized variety with $\dim X=n$. 
	Suppose that $\dim \mathrm{im}\,\Phi_{|\mathcal{L}|}=n$, $B:=\mathrm{Bs}|L|$ is finite and $X$ has only Cohen Macaulay singularities at each point of $B$. Suppose in addition that $d(X, \mathcal{L})\ge 2\Delta(X, \mathcal{L})-1$ if $\mathrm{char}\,k>0$. 
	
	Then $(X, \mathcal{L})$ has a ladder. 
	
\end{prop}
\begin{definition}(\cite[(1.5.12) (5.13)]{Fuj90}, or \cite[Ex I.5.12]{Har77})
	
	Let $M$ be a subvariety of $\mathbb{P}$ and $L$ be a linear subspace of $\mathbb{P}$ such that $M\cap L=\emptyset$. 
	A {\bf generalized cone} over $M$ with axis $L$ is defined as 
	$$M*L:=\bigcup\limits_{x\in M,y\in L} x*y,$$ 
	where $x*y$ is a line passing through $x$ and $y$.
\end{definition}
\begin{thm}(\cite[Theorem (1.5.10), (1.5.15)]{Fuj90})\label{delta0}
	
	Let $(X, \mathcal{L})$ be a polarized variety with $\Delta(X, \mathcal{L})= 0$ and $n=\dim X\ge 2$. 
	
	If $X$ is smooth, then $(X, \mathcal{L})$ is isomorphic to 
	\begin{itemize}
		\item[1)] $\left(\mathbb{P}^{\mathrm{n}}, \mathcal{O}_{\mathbb{P}^{\mathrm{n}}}(1)\right)$,
		\item[2)] $\left(C, \mathcal{O}_{C}(1)\right)$, where $C$ is a hyperquadric in $\mathbb{P}^{\mathrm{n}+1}$  and $\mathcal{O}_{C}(1)=\mathcal{O}_{ \mathbb{P}^{\mathrm{n}+1}}(1)|_C$,
		\item[3)] $\left(\mathbb{P}(\mathcal{E}), \mathcal{O}(1)\right)$, where $\mathbb{P}(\mathcal{E})$ is the scroll of a vector bundle on $\mathbb{P}^1$ which is a direct sum of line bundles of positive degrees and $\mathcal{O}(1)$ is the tautological line bundle, or
		\item[4)] $\left(\mathbb{P}^2, \mathcal{O}_{\mathbb{P}^2}(2)\right)$ (the Veronese surface).
	\end{itemize}
	
	If $X$ is singular, then $(X, \mathcal{L})$ is a generalized cone over a smooth variety $M$ with $\Delta(M,\mathcal{L}_M)=0$. 
\end{thm}
\begin{remark}
	The four classes above in  the smooth case are disjoint except for the case $n=d(X, \mathcal{L})=2$ where type 2) and type 3) coincide.
\end{remark}

\begin{lemma}\label{bpf2}
Let $(X, \mathcal{L})$ be a normal polarized variety with $\dim X=2$ and $\Delta(X, \mathcal{L})=1$. If $|\mathcal{L}|$ is not composed with a pencil, i.e., $\dim \Phi_{|\mathcal{L}|}(X)=2$, then $|\mathcal{L}|$ is base-point-free.
\end{lemma}
\begin{proof}
Since $h^0(X,\mathcal{L})\ge 3$, we may assume $(\mathcal{L})^2\ge2$.
Note that $\mathrm{Bs}|\mathcal{L}|$ is finite since $1=\Delta(X, \mathcal{L})\ge \dim \mathrm{Bs}|\mathcal{L}|+1$ by Theorem \ref{fujita}. 
By Prop \ref{ladder}, $(X, \mathcal{L})$ has a ladder $D\in |\mathcal{L}|$. The sectional genus $g:=g(X, \mathcal{L})=g(D, \mathcal{L}|_D)=h^1(D,\mathcal{O}_D)\ge 0$.

First we show that $g\not=0$. Otherwise, $D$ is a smooth rational curve. By Prop \ref{sectionalgenus}, we have $\Delta(X, \mathcal{L})=0$, a contradiction.  
Therefore,  $g> 0$,  and then $|\mathcal{L}|$ is base-point-free by Thm \ref{bpf}.

\end{proof}

\section{reducible polarized demi-normal schemes}

Let $X$ be a connected scheme, $\mathcal{L}$  be an invertible sheaf on $X$, and $C$ is a reduced subscheme. 
Denote by $\mathcal{R}_{X\to C,\mathcal{L}}$ the restriction map $H^0(X,\mathcal{L})\to H^0(C,\mathcal{L}|_C)$ and $r_{X\to C}(\mathcal{L}):=\dim \mathrm{im} \mathcal{R}_{X\to C,\mathcal{L}}$. We also use $\mathcal{R}_{X\to C}$ to denote $\mathcal{R}_{X\to C,\mathcal{L}}$ for simplicity if $\mathcal{L}$ is clear from the context. 

\begin{lemma}\label{restsecions}
Let $X$ be a proper pure-dimensional connected $S_2$ scheme, let $\mathcal{L}$ be an invertible sheaf with $\dim \mathrm{Bs} |\mathcal{L}|<\dim X-1$ and let $C$ be a reduced subscheme of codimension one on $X$. 

Then
\begin{align*}
r_{X\to C}(\mathcal{L})=\dim <\Phi_{|\mathcal{L}|}(C)>+1,
\end{align*}
where $<\Phi_{|\mathcal{L}|}(C)>$ is the projective subspace of $|\mathcal{L}|^*$ spanned by $\Phi_{|\mathcal{L}|}(C)$.

\end{lemma}
\begin{proof}
Denote $\mathbb{P}:=|\mathcal{L}|^*$. Since $\dim \mathrm{Bs} |\mathcal{L}|<\dim X-1=\dim C$,  $\Phi_{|\mathcal{L}|}$ is well-defined on $C^{\circ}:=C-\mathrm{Bs} |\mathcal{L}|$. Thus, $\Phi_{|\mathcal{L}|}(C):=\overline{\Phi_{|\mathcal{L}|}(C^{\circ})}$ is well-defined.
We have the following commutative diagram:
\[ \xymatrix{
      H^0(\mathbb{P},\mathcal{O}_{\mathbb{P}}(1))\ar[rr]^>>>>>>>>>{\mathcal{R}_{\mathbb{P}\to \Phi_{|\mathcal{L}|}(C)}}\ar[d]^{\cong}_{\Phi_{|\mathcal{L}|}^*} & & H^0(\Phi_{|\mathcal{L}|}(C),\mathcal{O}_{\mathbb{P}}(1)|_{\Phi_{|\mathcal{L}|}(C)})\ar@{^{(}->}[d]_{\Phi_{|\mathcal{L}|}^*}  \\
    H^0(X,\mathcal{L})\ar[rr]^{\mathcal{R}_{X\to C}}& & H^0(C,\mathcal{L}|_C)
   }.
\]
Therefore $r_{X\to C}(\mathcal{L})=r_{\mathbb{P}\to \Phi_{|\mathcal{L}|}(C)}(\mathcal{O}_{\mathbb{P}}(1))=\dim <\Phi_{|\mathcal{L}|}(C)>+1$.
\end{proof}

\begin{thm}\label{normal&divisor}
Let $(X, \mathcal{L})$ be a normal polarized variety and let $C$ be a reduced subscheme of codimension 1.

Then \[\Delta(X, \mathcal{L})+r_{X\to C}(\mathcal{L})-\dim X\ge 0.\]
\end{thm}
\begin{proof}
Denote $n:=\dim X$. The case $n=1$ is trivial, so we may assume $n\ge 2$. 

We only need to consider the case $\Delta(X, \mathcal{L})<n$. 
Therefore we may assume $\dim \mathrm{Bs}|\mathcal{L}|\le n-2$ by Thm \ref{fujita}.
We may assume further $h^0(X,\mathcal{L})\ge2$ (otherwise, $\Delta(X, \mathcal{L})+r_{X\to C}(\mathcal{L})-\dim X\ge\Delta(X, \mathcal{L})-n=(\mathcal{L})^n-h^0(X,\mathcal{L})\ge 0$).

Let $\pi:\tilde{X}\rightarrow X$ be a composition of minimal resolution of singularities and blowups such that $\pi^*|\mathcal{L}|= |M|+F$, where $F$ is the fixed part and $|M|$ is the movable part which is base-point-free.

Let $W$ be the image of $\tilde{X}$ under $\Phi_{|M|}$ and $\phi$ be the induced map of $\Phi_{|M|}$ onto $W$. We have the following commutative diagram:
\[ \xymatrix{
    \tilde{X} \ar[r]^{\pi}\ar[d]_{\phi}\ar[dr]^{\Phi_{|M|}} & X \ar@{-->}[d]^{\Phi_{|\mathcal{L}|}}\\
W\ar@{^{(}->}[r]& \mathbb{P}^{h^0(X,\mathcal{L})-1}.
   }
\]

We claim that $(\mathcal{L})^n\ge \deg  W$.
First we have $\pi_*(F\cdot \pi^*\mathcal{L})=\pi_*(F)\cdot\mathcal{L}=0$, 
since $\pi(F)=\mathrm{Bs}|\mathcal{L}|$ is of codimension at least 2.
If $\dim W=\dim X=n$, $(\mathcal{L})^n=(\pi^*\mathcal{L})^n=M^n+M^{n-1}\cdot F\ge M^n=\deg\phi \,\, \deg W
\ge \deg  W$. If $m:=\dim W<\dim X$, $(\mathcal{L})^n=M^m\cdot(\pi^*\mathcal{L})^{n-m}=(\deg W \cdot Z)\cdot (\pi^*\mathcal{L})^{n-m}\ge \deg W$, where $Z$ is a general fiber of $\phi$. Hence the claim is true.

Therefore, $\Delta(X, \mathcal{L})=(\mathcal{L})^n-h^0(X,\mathcal{L})+n\ge \deg  W-h^0(W,\mathcal{O}_W(1))+n=\Delta(W,\mathcal{O}_W(1))+n-\dim  W$.
Hence $\Delta(X, \mathcal{L})+r_{X\to C}(\mathcal{L})-n\ge \Delta(W,\mathcal{O}_W(1))+r_{X\to C}(\mathcal{L})-\dim  W$.
Next we show that $r_{X\to C}(\mathcal{L})\ge \dim  W$. 

We claim that the image $\Phi_{|\mathcal{L}|}(C)$ of $C$ is of codimension  equal or less than 1 in $W$.
For the case $n=2$ and $|\mathcal{L}|$ is composed with a pencil, i.e., $\dim \Phi_{|\mathcal{L}|}(X)=1$, the claim is trivial. For the case $n=2$ and $|\mathcal{L}|$ is not composed with a pencil, $|\mathcal{L}|$ is base-point-free by Lemma \ref{bpf2}. Hence $\Phi_{|\mathcal{L}|}$ is a morphism and it contracts no curve as $\mathcal{L}$ is ample. Therefore $\Phi_{|\mathcal{L}|}(C)$ is of codimension 1 in $W$.
For the case $n\ge3$, suppose for a contradiction that $\Phi_{|\mathcal{L}|}(C)$ is of codimension greater than 1 in $W$.  Thus we have $\Phi_{|\mathcal{L}|*}C=0$ and then $\phi_*(\bar C)=0$, where $\bar C$ is the strict transformation of $C$ in $\tilde{X}$. 
We would have $0< C\cdot(\mathcal{L})^{n-1}=\bar C \cdot(\pi^*\mathcal{L})^{n-1}=\bar{C}\cdot M\cdot (\pi^*\mathcal{L})^{n-2}+\bar{C}\cdot F\cdot (\pi^*\mathcal{L})^{n-2}=0$, a contradiction. Thus we have proven the claim.


By Lemma \ref{restsecions} we have $r_{X\to C}(\mathcal{L})=\dim \langle\Phi_{|\mathcal{L}|}(C)\rangle +1$. It is obvious that $\dim\langle\Phi_{|\mathcal{L}|}(C)\rangle\ge \dim \Phi_{|\mathcal{L}|}(C)$. Therefore $r_{X\to C}(\mathcal{L})\ge \dim \Phi_{|\mathcal{L}|}(C)+1\ge \dim  W$.


We conclude that $\Delta(X, \mathcal{L})+r_{X\to C}(\mathcal{L})-n\ge \Delta(W,\mathcal{O}_W(1))\ge0$.
\end{proof}
\begin{cor}\label{delta0gluing}
Let $(X,\mathcal{L})$ and $C$ be as in Thm \ref{normal&divisor}. Assume further $\Delta(X,\mathcal{L})=0$ and $r_{X\to C}(\mathcal{L})=\dim X$.
Then $C$ is a hyperplane of $X$ with respect to $\mathcal{L}$.
\end{cor}
\begin{proof}
First $\Delta(X,\mathcal{L})=0$ implies that $\mathcal{L}$ is very ample by Thm \ref{veryample}. As a result,  $\mathcal{L}|_C$ is very ample as well. 
By Lemma \ref{restsecions}, $r_{X\to C}(\mathcal{L})=\dim X$ implies that 
$\dim \langle\Phi_{|\mathcal{L}|}(C)\rangle=\dim X-1$. 
Therefore, $\langle \Phi_{|\mathcal{L}|}(C)\rangle\cong \mathbb{P}^{\dim X-1}\subset \mathbb{P}^{h^0(X,\mathcal{L})-1}$. 
However, since $\Phi_{|\mathcal{L}|}$ is an embedding, $\dim \Phi_{|\mathcal{L}|}(C)= \dim C=\dim X-1$.
Therefore, $\dim \Phi_{|\mathcal{L}|}(C)=\dim \langle\Phi_{|\mathcal{L}|}(C)\rangle$, 
which implies that $\Phi_{|\mathcal{L}|}(C)=\langle \Phi_{|\mathcal{L}|}(C)\rangle\cong \mathbb{P}^{\dim X-1}$. 
Therefore, $C\cong \Phi_{|\mathcal{L}|}(C)\cong \mathbb{P}^{\dim X-1}$ and $\mathcal{L}|_C\cong \mathcal{O}_{\mathbb{P}^{\dim X-1}}(1)$. 
Hence, $C$ is a hyperplane of $X$ with respect to $\mathcal{L}$.
\end{proof}
\begin{thm}\label{deminormal&divisor}
Let $(X, \mathcal{L})$ be an irreducible polarized non-normal variety.  Then \[\Delta(X, \mathcal{L})\ge 1.\]
Moreover, if $C$ is a reduced subscheme of codimension 1 which does not contain any component of the non-normal locus, 
then we have \[\Delta(X, \mathcal{L})+r_{X\to C}(\mathcal{L})-\dim X\ge 0.\]
\end{thm}
\begin{proof}
The first statement is a direct consequence of Thm \ref{fujita} and Thm \ref{veryample}. 

Let $\pi:\bar{X}\rightarrow X$ be the normalization morphism and let $\bar C$ be the proper transformation of $C$ in $\bar X$.
Since $H^0(X,\mathcal{L})\cong \pi^*H^0(X,\mathcal{L})\subset H^0(\bar{X},\pi^*\mathcal{L})$ and $(\mathcal{L})^{\dim X}=(\pi^*\mathcal{L})^{\dim \bar X}$, we have $\Delta(X,\mathcal{L})\ge \Delta(\bar{X},\pi^*\mathcal{L})\ge 0$. 

Next we have the following commutative diagram:

\[\xymatrix{
0 \ar[r] & H^0(\bar X,\pi^*\mathcal{L}-\bar C)\ar[r]
& H^0(\bar X,\pi^*\mathcal{L})\ar[r]^{\mathcal{R}_{\bar X\to\bar C}}& H^0(\bar C,\pi^*\mathcal{L}|_{\bar C})\\
0 \ar[r] & H^0(X,\mathcal{L}-C)\ar[r]\ar@{_(->}[u]_{\pi^*} & H^0(X,\mathcal{L})\ar[r]^{\mathcal{R}_{X\to C}}\ar@{_(->}[u]_{\pi^*}& H^0(C,\mathcal{L}|_C).
}
\]

Hence $h^0(X,\mathcal{L})-r_{X\to C}(\mathcal{L}) =h^0(X,\mathcal{L}-C)\le h^0(\bar X,\pi^*\mathcal{L}-\bar C)=h^0(\bar X,\pi^*\mathcal{L})-r_{\bar X\to \bar{C}}(\pi^*\mathcal{L})$.
Therefore $\Delta(X, \mathcal{L})+r_{X\to C}(\mathcal{L})-\dim X\ge \Delta(\bar X, \pi^*\mathcal{L})+r_{\bar X\to \bar{C}}(\pi^*\mathcal{L})-\dim \bar{X}\ge 0$, where the second inequality follows from Thm \ref{normal&divisor}.
\end{proof}

\begin{lemma}\label{increasing}
Let $(X,\mathcal{L})$ be a connected polarized demi-normal  scheme. Assume $X=X_1\cup X_2$ where $X_1$ is connected and $X_2$ is irreducible.  Then \[\Delta(X, \mathcal{L})\ge \Delta(X_1, \mathcal{L}|_{X_1}).\]

\end{lemma}
\begin{proof}
Let $C:=X_1\cap X_2$ be the connecting subscheme.
By \eqref{sectiongluing}, we have
\begin{align*}
\Delta(X, \mathcal{L})&\ge \Delta(X_1, \mathcal{L}|_{X_1})+\Delta(X_2, \mathcal{L}|_{X_2})+\max\{r_{X_1\to C}(\mathcal{L}|_{X_1}),r_{X_2\to C}(\mathcal{L}|_{X_2})\}-\dim X\\
&\ge \Delta(X_1, \mathcal{L}|_{X_1})+\Delta(X_2, \mathcal{L}|_{X_2})+r_{X_2\to C}(\mathcal{L}|_{X_2})-\dim X.
\end{align*}
By  Thm \ref{normal&divisor} and Thm \ref{deminormal&divisor}, $\Delta(X_2, \mathcal{L}|_{X_2})+r_{X_2\to C}(\mathcal{L}|_{X_2}))-\dim X\ge 0$ whenever $X_2$ is normal or not.
Therefore $\Delta(X, \mathcal{L})\ge \Delta(X_1, \mathcal{L}|_{X_1})$.
\end{proof}
As a result, we obtain the following theorem
\begin{thm}\label{mainthm}
Let $(X,\mathcal{L})$ be a connected polarized demi-normal  scheme.
Then  \[\Delta(X, \mathcal{L})\ge 0.\]
\end{thm}

We say that a connected demi-normal  scheme $X$ is a {\bf tree of subschemes} $X_i$, if $X=\cup X_i$ and $X_i\cap X_j$ is either an irreducible subscheme of codimension one or a subscheme of codimension $\ge 2$ for $i\not =j$.
This is equivalent to say that the dual graph of codimension $\le1$ points of $X$ is a tree. 

Let $X=\cup X_i$ be a connected demi-normal scheme, and let $C_{ij}:=X_i\cap X_j$ be the connecting subscheme of $X_i$ and $X_j$. 
Let $C:=\sum\limits_{\mathrm{codim}_XC_{ij}=1} C_{ij}$. Normalizing $X$ along $C$, we obtain $(\bar{X}:=\bigsqcup X_i, \bar{C}:=\sum\bar{C}_{ij})=\bigsqcup(X_i, \sum_j\bar{C}_{ij})$. Denote by $\nu:\bar{X}\to X$ the normalization morphism along $C$. The morphism $\nu$ induces an isomorphism $\phi_{ij}:\bar{C}_{ij}\stackrel{\cong}{\to} \bar{C}_{ji}$. Conversely, $X$ can be viewed as being constructed from $\bigsqcup(X_i, \sum\bar{C}_{ij})$ by gluing $\bar{X}$ along $\bar{C}$ via $\phi_{ij}$'s (see \cite[chapter 5]{KollarSMMP}). 

\begin{thm} \label{delta0schemes}
	Let $(X,\mathcal{L})$ be a connected polarized demi-normal scheme with \newline $\Delta(X, \mathcal{L})=0$. 
	 
 	Then $X$ is a tree of subschemes $X_i$'s with $\Delta(X_i,\mathcal{L}|_{X_i})=0$ glued along hyperplanes. 
\end{thm}
\begin{proof}
	By Lemma \ref{increasing}, we see that $\Delta(X_i,\mathcal{L}|_{X_i})=0$ for each $(X_i,\mathcal{L}|_{X_i})$. 
	
	Assume that $C_{ij}:=X_i\cap X_j$ is a subscheme of codimension one. By Lemma \ref{increasing}, we have 
	$\Delta(X_i\cup X_j,\mathcal{L}|_{X_i\cup X_j})=\Delta(X_i,\mathcal{L}|_{X_i})=\Delta(X_j,\mathcal{L}|_{X_j})=0$. 
	By the proof of Lemma \ref{increasing}, we have $r_{X_i\to C_{ij}}(\mathcal{L}|_{X_i}))=r_{X_j\to C_{ij}}(\mathcal{L}|_{X_j}))=\dim X$. 
	By Cor \ref{delta0gluing}, we see that $C_{ij}$ is a hyperplane of $X_i$ (resp. $X_j$) with respect to $\mathcal{L}|_{X_i}$ (resp. $\mathcal{L}|_{X_j}$). 
	Therefore, the statement follows.
\end{proof}
\begin{remark}
	Polarized varieties $(X_i,\mathcal{L}_i)$  with $\Delta(X_i,\mathcal{L}_i)=0$ have been classified in Thm \ref{delta0}. 
	We are able to describe hyperplanes $C_{ij}$ on $(X_i,\mathcal{L}_i)$. 
	In order to construct a connected  polarized demi-normal scheme $(X,\mathcal{L})$ from $(X_i,\mathcal{L}_i, C_{ij})$'s, we may first embed $X_i$ to $\mathbb{P}_i$ via $\Phi_{|\mathcal{L}_i|}$ and then embed  $\mathbb{P}_i$'s into a big projective space $\mathbb{P}$. 
	The hyperplanes $C_{ij}$'s are linear subspaces in $\mathbb{P}$ of the same dimension, so they are isomorphic to each other. 
	Thus, we can glue $X_i$'s along $C_{ij}$'s in $\mathbb{P}$ and obtain the desired $X$ and $\mathcal{L}=\mathcal{O}_X(1)$. 
\end{remark}


\begin{cor}

Let $(X,\Lambda)$ be a connected KSBA stable log scheme with Cartier index $I$. Then \[\Delta(X, I(K_X+\Lambda))\ge 0.\]
\end{cor}

\begin{remark}\label{construction}
	When $I=1$, $\mathrm{dim} X=2$ and $\Lambda$ is reduced, $\Delta(X, K_X+\Lambda)=(K_X+\Lambda)^2-p_g(X,\Lambda)+2\ge 0$ is just the stable log Noether inequality in \cite{LR13}. 
\end{remark}
%

In the following, we will construct a KSBA stable log scheme with Cartier index $I=1$ and $\Delta(X, K_X+\Lambda)= 0$. 
To avoid repetition, we only give examples with $2m+1$ irreducible components.
The construction of KSBA stable log scheme with $2m$ irreducible components is similar. 
These examples demonstrate that when $I=1$, the inequality $\Delta(X, K_X+\Lambda)\ge 0$ is sharp.  
\begin{example}\label{exmp}
	Let $B=H_1+...+H_{n+2}$ be a snc divisor consisting of hyperplanes on $\mathbb{P}^n$. We have $K_{\mathbb{P}^n}+B\sim \mathcal{O}_{\mathbb{P}^n}(1)$ and $(\mathbb{P}^n,B)$ is lc.
	
	\begin{align*}
	\text{Let}\quad &(\overline{X_1},\overline{D_{1,2}}+\overline{\Lambda_1})\cong (\mathbb{P}^n,H_1+(H_2+...+H_{n+2})),  \\
	&(\overline{X_{k}},\overline{D_{k,k-1}}+\overline{D_{k,k+1}}+\overline{\Lambda_{k}})\cong (\mathbb{P}^n,H_1+H_2+(H_3+...+H_{n+2})), k=2,...,2m,  \\
	&(\overline{X_{k}},\overline{D_{k,k-1}}+\overline{D_{k,k+1}}+\overline{\Lambda_{k}})\cong (\mathbb{P}^n,H_2+H_1+(H_3+...+H_{n+2})), k=3,...,2m-1, \quad\text{and} \\
	&(\overline{X_{2m+1}},\overline{D_{2m+1,2m}}+\overline{\Lambda_{2m+1}})\cong (\mathbb{P}^n,H_2+(H_1+...+H_{n+2})).
	\end{align*}

	We use the method in Remark \ref{construction} to construct the polarized demi-normal scheme.
	We glue $\overline{X_{2i-1}}$ and $\overline{X_{2i}}$ by the isomorphism $\overline{D_{2i-1,2i}}\cong H_1\cong \overline{D_{2i,2i-1}}$ induced from $\overline{X_{2i-1}}\cong \mathbb{P}^n\cong \overline{X_{2i}}$, $i=1,...,m$.  Similarly, we glue $\overline{X_{2i}}$ and $\overline{X_{2i+1}}$ by the isomorphism $\overline{D_{2i,2i+1}}\cong H_2\cong \overline{D_{2i+1,2i}}$, $i=1,...,m$. 
	Then we obtain a  demi-normal log scheme $X$ and a boundary divisor $\Lambda$ which is the image of $\sum \overline{\Lambda_i}$. The divisor $K_X+\Lambda$ is a polarization on $X$ with $\Delta(X, K_X+\Lambda)= 0$. Note that $(K_X+\Lambda)|_{X_i}\cong \mathcal{O}_{\mathbb{P}^n}(1)$.  
	By Thm 5.38 in \cite{KollarSMMP} the singularities of the log scheme $(X,\Lambda)$ are slc. Thus $(X,\Lambda)$ is a KSBA stable log scheme with Cartier index one and $\Delta(X, K_X+\Lambda)= 0$.
	Moreover, $( K_X+\Lambda)^n=2m+1$. See Figure \ref{fig:n2} for the $n=2$ case.

\begin{figure}[h!]
	\centering
		%
		%
	\begin{tikzpicture}[font=\tiny]
		
		\tikzset{declare function={
				mght(\x,\th,\m)=\x* tan(\th)+\m;
				arctheta(\h,\d)=atan (\h/\d);
			}
		}
		\begin{scope}[xshift=0cm,xscale=1]
			\coordinate (O1) at (2,2);  
			\draw[fill=blue, opacity=0.2, rotate around={0:(O1)}] (O1) ellipse (1.3 and 1.8);
			\node[left] at (0.5,2){$\mathbb{P}^2$};
			\draw (1.7,0.5)--(2.6,3.3); 
			\node[right] at (2.2,1.7){$H_1$};
			\draw (2.3,0.5)--(1.4,3.3); 
			\node[left] at (1.8,1.7){$H_2$};
			\draw (1.1,3)--(2.8,3); 
			\node[right] at (3,3.2){$H_3$};
			\draw (1,2.2)--(2.9,2.6);
			\node[right] at (3.2,2.5){$H_4$};

		\end{scope}
		\begin{scope}[xshift=10cm,xscale=1]
			\tikzmath{ \r = 4; 
				\th=10;
				\x=4; 
				\d=3;
				\m=0.3;
				\h=\x* tan \th+\m;
				\tht=atan (\h/\r);
			};  
			\coordinate (O) at (0,0);
			
			\tikzmath{\rt=140;}

			\tikzmath{\y1= mght(5,\th,\m);
				\tt1=arctheta(\y1,8-5);
				\y2= mght(2,\th,\m);
				\tt2=arctheta(\y2,8-2);}
			\fill[blue,opacity=0.2,rotate around={\rt:(O)}] (5, \y1) arc [radius=\y1/sin \tt1,start angle=180-\tt1, end angle=180+\tt1] --(2, -\y2) arc [radius=\y2/sin \tt2,start angle=180+\tt2, end angle=180-\tt2]--cycle; 
			\draw[rotate around={\rt:(O)}] ({180-\th}:0.5)-- (-\th:4.8);
			\draw[rotate around={\rt:(O)}] ({180+\th}:0.5)-- (\th:4.8);
			\tikzmath{\m=0.2;\y1= mght(4,\th,\m);
				\tt1=arctheta(\y1,8-4);
				\y2= mght(3,\th,\m);
				\tt2=arctheta(\y2,8-3);}
			\draw[rotate around={\rt:(O)}] (4, \y1) arc [radius=\y1/sin \tt1,start angle=180-\tt1, end angle=180+\tt1];
			\draw[rotate around={\rt:(O)}] (3, -\y2) arc [radius=\y2/sin \tt2,start angle=200+2*\tt2, end angle=200];

			\tikzmath{\rt=120;}

			\tikzmath{\y1= mght(5,\th,\m);
				\tt1=arctheta(\y1,8-5);
				\y2= mght(2,\th,\m);
				\tt2=arctheta(\y2,8-2);}
			\fill[blue,opacity=0.2,rotate around={\rt:(O)}] (5, \y1) arc [radius=\y1/sin \tt1,start angle=180-\tt1, end angle=180+\tt1] --(2, -\y2) arc [radius=\y2/sin \tt2,start angle=180+\tt2, end angle=180-\tt2]--cycle; 
			\tikzmath{\m=0.2;\y1= mght(4,\th,\m);
				\tt1=arctheta(\y1,8-4);
				\y2= mght(3,\th,\m);
				\tt2=arctheta(\y2,8-3);}
			\draw[thick,red,rotate around={\rt:(O)}] ({180-\th}:0.5)-- (-\th:4.8);
			\draw[thick,red,rotate around={\rt:(O)}] ({180+\th}:0.5)-- (\th:4.8);
			\draw[rotate around={\rt:(O)}] (4, \y1) arc [radius=\y1/sin \tt1,start angle=180-\tt1, end angle=180+\tt1];
			\draw[rotate around={\rt:(O)}] (3, \y2) arc [radius=\y2/sin \tt2,start angle=160-2*\tt2, end angle=160];

			\tikzmath{\rt=30;}

			\tikzmath{\y1= mght(5,\th,\m);
				\tt1=arctheta(\y1,8-5);
				\y2= mght(2,\th,\m);
				\tt2=arctheta(\y2,8-2);}
			\fill[blue,opacity=0.2,rotate around={\rt:(O)}] (5, \y1) arc [radius=\y1/sin \tt1,start angle=180-\tt1, end angle=180+\tt1] --(2, -\y2) arc [radius=\y2/sin \tt2,start angle=180+\tt2, end angle=180-\tt2]--cycle; 
			\draw[rotate around={\rt:(O)}] ({180-\th}:0.5)-- (-\th:4.8);
			\draw[thick,red,rotate around={\rt:(O)}] ({180+\th}:0.5)-- (\th:4.8);
			\tikzmath{\m=0.2;\y1= mght(4,\th,\m);
				\tt1=arctheta(\y1,8-4);
				\y2= mght(3,\th,\m);
				\tt2=arctheta(\y2,8-3);}
			\draw[rotate around={\rt:(O)}] (4, \y1) arc [radius=\y1/sin \tt1,start angle=180-\tt1, end angle=180+\tt1];
			\draw[rotate around={\rt:(O)}] (3, -\y2) arc [radius=\y2/sin \tt2,start angle=200+2*\tt2, end angle=200];
			
			\node[above] at (140:5.2) {$X_{1}$};
			\node[above] at (130:5.3) {$D_{12}$};
			\node[left] at (143:3.2) {$\Lambda_{1}$};
			\node[above] at (120:5.2) {$X_{2}$};
			\node[above] at (110:5.3) {$D_{23}$};
			\node[left] at (113:3.2) {$\Lambda_{2}$};
			\node[right] at (30:5.2) {$X_{2m+1}$};
			\node[above] at (40:5.3) {$D_{2m,2m+1}$};
			\node[right] at (29:2.8) {$\Lambda_{2m+1}$};
			\draw[dashed] (100:3.5) arc [radius=3.5,start angle=100, end angle=50];
			
		\end{scope}
	\end{tikzpicture}
	\caption{$n=2$ case.} \label{fig:n2}
\end{figure}

\end{example}

%

\bibliographystyle{alpha}

\begin{thebibliography}{plain}

%

\bibitem[Fuj90]{Fuj90}
T. Fujita.
\newblock {\em Classification Theories of Polarized Varieties}.
\newblock Cambridge University Press, 1990.

\bibitem[Har77]{Har77}
R. Hartshorne. 
\newblock {\em Algebraic Geometry}. 
\newblock Springer New York, 1977.

\bibitem[Kol13]{KollarSMMP}
J. Koll{\'a}r.
\newblock {\em Singularities of the minimal model program}, volume 200 of 
  Cambridge Tracts in Mathematics.
\newblock Cambridge University Press, Cambridge, 2013.
\newblock With a collaboration of S{\'a}ndor Kov{\'a}cs.

\bibitem[LR13]{LR13}
W. Liu and S. Rollenske.
\newblock {\em Geography of Gorenstein stable log surfaces[J]}.
\newblock  Transactions of the American Mathematical Society, 2013, 368(4), 2563-2588.

%
%
%
%


\end{thebibliography}

\end{document}